\documentclass[12pt]{amsart}
\pdfoutput=1
\usepackage[headings]{fullpage}
\usepackage{amssymb}
\usepackage{graphicx}
\usepackage{url}
\usepackage[bookmarks=true,%
    colorlinks=true,%
    linkcolor=blue,%
    citecolor=blue,%
    filecolor=blue,%
    menucolor=blue,%
    urlcolor=blue,%
    breaklinks=true]{hyperref}

\newtheorem{theorem}{Theorem}[section]
\theoremstyle{definition}

\newtheorem{lemma}[theorem]{Lemma}
\newtheorem{definition}[theorem]{Definition}
\newtheorem{remark}[theorem]{Remark}
\newtheorem{corollary}[theorem]{Corollary}

\def\BN{\mathbb N}
\def\BZ{\mathbb Z}

\def\BQ{\mathbb Q}
\def\BR{\mathbb R}

\def\calS{\mathcal S}

\def\d{\delta}
\def\b{\beta}

\def\longto{\longrightarrow}

\def\js{\mathrm{js}}
\def\jsr{\mathrm{js}_{\BR}}
\def\bs{\mathrm{bs}}

\def\Tet{\mathrm{Tet}}
\def\U{\mathrm{U}}
\def\lt{\mathrm{lt}}

\newcommand{\qbinom}[2]{\genfrac{[}{]}{0pt}{}{#1}{#2}}

\begin{document}


\title[Quadratic integer programming and the slope conjecture]{Quadratic 
integer programming and the slope conjecture}
\author{Stavros Garoufalidis}
\address{School of Mathematics \\
         Georgia Institute of Technology \\
         Atlanta, GA 30332-0160, USA \newline 
         {\tt \url{http://www.math.gatech.edu/~stavros }}}
\email{stavros@math.gatech.edu}
\author{Roland van der Veen}
\address{Mathematisch Intstituut \\
Leiden University, Leiden \\
Niels Bohrweg 1\\
The Netherlands 
\newline 
{\tt \url{http://www.rolandvdv.nl}}}
\email{r.i.van.der.veen@math.leidenuniv.nl}
\thanks{
1991 {\em Mathematics Classification.} Primary 57N10. Secondary 57M25.
}

\date{August 1, 2016} 

\begin{abstract}
The Slope Conjecture relates a quantum knot invariant, (the degree
of the colored Jones polynomial of a knot) with a classical one 
(boundary slopes of incompressible surfaces in the
knot complement). 

The degree of the colored Jones polynomial can be computed
by a suitable (almost tight) state sum and the solution of a corresponding
quadratic integer programming problem. We illustrate this principle
for a 2-parameter family of 2-fusion knots.
Combined with the results of Dunfield and the first author, this confirms 
the Slope Conjecture for the 2-fusion knots of one sector.
\end{abstract}

\maketitle

{\footnotesize
\tableofcontents
}




\section{Introduction}
\label{sec.intro}

\subsection{The Slope Conjecture}
\label{sub.cj}

The Slope Conjecture of~\cite{Ga1} relates a quantum knot invariant, 
(the degree of the colored Jones polynomial of a knot) with a classical one 
(boundary slopes of incompressible surfaces in the
knot complement). The aim of our paper is to compute the degree of the
colored Jones polynomial of a 2-parameter family of 2-fusion knots using
methods of tropical geometry and quadratic integer programming, and
combined with the results of~\cite{DG}, to confirm the Slope Conjecture
for a large class of 2-fusion knots.

Although the results of our paper concern an identification of a
classical and a quantum knot invariant they require no prior knowledge of 
knot theory nor familiarity with incompressible surfaces or the colored 
Jones polynomial of a knot or link. As a result, we will not recall the
definition of an {\em incompressible surface} of a 3-manifold with torus
boundary, nor definition of the {\em Jones polynomial} 
$J_L(q) \in \BZ[q^{\pm 1/2}]$ of a knot or link $L$ in 3-space. These
definitions may be found in several 
texts~\cite{Ha,HO} and ~\cite{Jo,Tu,Tu:book,Kauffman}, respectively. 
A stronger quantum invariant is the {\em colored Jones polynomial}
$J_{L,n}(q) \in \BZ[q^{\pm 1/2}]$, where $n \in \BN$, which is a linear 
combination of the Jones polynomial of a link and its parallels 
\cite[Cor.2.15]{Kirby}. 

To formulate the Slope Conjecture, let $\d_K(n)$ denote the
$q$-degree of the colored Jones polynomial $J_{K,n}(q)$. It is known that
$\d_K: \BN \longto \BQ$ is a {\em quadratic quasi-polynomial}~\cite{Ga2} for
large enough $n$. In other words, for large enough $n$ we have
$$
\d_K(n)=c_{K,2}(n) n^2 + c_{K,1}(n) n + c_{K,0}(n)
$$
where $c_{K,j}: \BN \longto \BQ$ are periodic functions. 
The {\em Slope Conjecture} states that the finite set of values of $4 c_{K,2}$ 
is a subset of the set $\bs_K$ of slopes of boundary incompressible surfaces 
in the knot complement. 
The set of values of $c_{K,2}$ is referred to as the {\em Jones slopes} of the 
knot $K$. In case $c_{K,2}$ is constant, as often the case, it is called the 
Jones slope, abbreviated $\js_K$. At the time of writing no knots with more 
than one Jones slope are known to the authors.


\subsection{Boundary slopes}
\label{sub.bsk}

In general there are infinitely many non-isotopic 
boundary incompressible surfaces in the complement of a knot $K$. However, 
the set $\bs_K$ of their boundary slopes is always a nonempty finite subset
of $\BQ\cup\{\infty\}$~\cite{Ha}. The set of boundary slopes is algorithmically
computable for the case of Montesinos knots (by an algorithm of 
Hatcher-Oertel~\cite{HO}; see also~\cite{Du}) and for the case of alternating
knots (by Menasco~\cite{Menasco}) where incompressible surfaces can often be 
read from an alternating planar projection. The $A$-polynomial of a knot
determines some boundary slopes~\cite{CCGLS}. However, the $A$-polynomial
is difficult to compute, for instance it is unknown for the alternating
Montesinos knot $9_{31}$~\cite{Culler:apoly}.
Other than this, it is unknown how to produce a single non-zero boundary 
slope for a general knot, or for a family of them.

\subsection{Jones slopes, state sums and quadratic integer programming}
\label{sub.trop.deg}

There are close relations between linear programming, normal
surfaces and their boundary slopes. It is less known that that the degree of 
the colored Jones polynomial is closely related to {\em tropical geometry} and 
{\em quadratic integer programming}. The key to this relation is 
a state sum formula for the colored Jones polynomial. State sum formulas
although perhaps unappreciated, are abundant in quantum topology. 
A main point of~\cite{GL1} is that state sums imply $q$-holonomicity.
Our main point is that under some fortunate circumstances,
state sums give effective formulas for their $q$-degree. 
To produce state sums in quantum topology, one may use
\begin{itemize}
\item[(a)]
a planar projection of a knot and an $R$-matrix~\cite{Tu,Tu:book},
\item[(b)]
a shadow presentation of a knot and quantum $6j$-symbols and 
$R$-matrices~\cite{Tu:shadow,Co,CT},
\item[(c)]
a fusion presentation of a knot and quantum 
$6j$-symbols~\cite{Thurston:shadow,V,GV}.
\end{itemize}
All those state sum formulas are obtained by contractions of tensors
and in the case of the colored Jones polynomial, lead to an expression of 
the form:
\begin{equation}
\label{eq.JSnk}
J_{K,n}(q)=\sum_{k \in n P \cap \BZ^r} S(n,k)(q)
\end{equation}
where 
\begin{itemize}
\item
$n$ is a natural number, the color of the knot,
\item
$P$ is a rational convex polytope such that the lattice points $k$ of $n P$
are the admissible states of the state sum, 
\item
the summand $S(n,k)$ is a product of weights of building blocks. The weight
of a building block is a rational function of $q^{1/4}$ and its $q$-degree
is a piece-wise quadratic function of $(n,k)$.
\end{itemize}
Let $\d(f(q))$ denote the $q$-degree of a rational function 
$f(q) \in \BQ(q^{1/4})$. This is defined as follows: if $f(q)=a(q)/b(q)$
where $a(q),b(q) \in \BQ[q^{1/4}]$ with $b(q) \neq 0$, then 
$\d(f(q))=\d(a(q))-\d(b(q))$, with the understanding that when $a(q)=0$,
then $\d(a(q))=-\infty$. It is easy to see that the $q$-degree of a rational
function $f(q) \in \BQ(q^{1/4})$ is well-defined and satisfies the elementary
properties
\begin{subequations}
\begin{align}
\label{eq.tg1}
\d(f(q)  g(q) ) & =\d(f(q) )+\d(g(q) ) 
\\
\label{eq.tg2}
 \d(f(q) +g(q) ) & \leq \max\{ \d(f(q)), \d(g(q)) \}
\end{align}
\end{subequations}
The state sum~\eqref{eq.JSnk} together with the above identities 
implies that the degree $\d(n,k)$ of $S(n,k)(q)$ is a piece-wise quadratic 
polynomial in $(n,k)$. Moreover, if there is no cancellation in the leading 
term of Equation~\eqref{eq.JSnk} (we will call such formulas {\em tight}),
it follows that the degree $\d_K(n)$ of the colored Jones polynomial
$J_{K,n}(q)$ equals to $\hat \d(n)$ where
\begin{equation}
\label{eq.dkn}
\hat \d(n)=\max\{ \d(n,k) \,\, | \,\, k \in n P \cap \BZ^r \}
\end{equation}
Computing $\hat \d(n)$ is a problem in quadratic integer programming 
(in short, QIP)~\cite{LORW,Onn,LHORW,KPor}. 

The answer is given by a quadratic quasi-polynomial of 
$n$, whose coefficient of $n^2$ is independent of $n$, 
for all but finitely many $n$. 
If we are interested in the quadratic part of $\hat \d(n)$, then we can
use state sums for which the degree of the sum drops by the maximum
degree of the summand by at most a linear function of $n$. We will call
such state sums {\em almost tight}.

A related and simpler real optimization problem is the following
\begin{equation}
\label{eq.drkn}
\hat \d_{\BR}(n)=\max\{ \d(n,x) \,\, | \,\, x \in n P \}
\end{equation}
Using a change of variables $x=ny$, it is easy to see that $\hat \d_{\BR}(n)$
is a quadratic polynomial of $n$, for all but finitely many $n$.

Thus, an almost tight state sum for the colored Jones polynomial
a knot (of even more, of a family of knots) allows us to compute the degree 
of their colored Jones polynomial using QIP. Our main point is that
it is easy to produce tight state sums using fusion, and in the case they
are almost tight, it is possible to analyze ties and cancellations.
We illustrate in Theorem~\ref{thm.1} below for the 2-parameter family of 
2-fusion knots.

\subsection{$2$-fusion knots}
\label{sub.qt2fusion}

Consider the 3-component seed link $K$ as in Figure~\ref{f.seedlink}
and the knot $K(m_1,m_2)$ obtained by $(-1/m_1,-1/m_2)$ filling on $K$
for two integers $m_1,m_2$. $K(m_1,m_2)$ is the 2-parameter family of 2-fusion
knots. This terminology is explained in detail in Section~\ref{sec.kfusion}.


\begin{figure}[!htpb]
\begin{center}
\includegraphics[height=0.15\textheight]{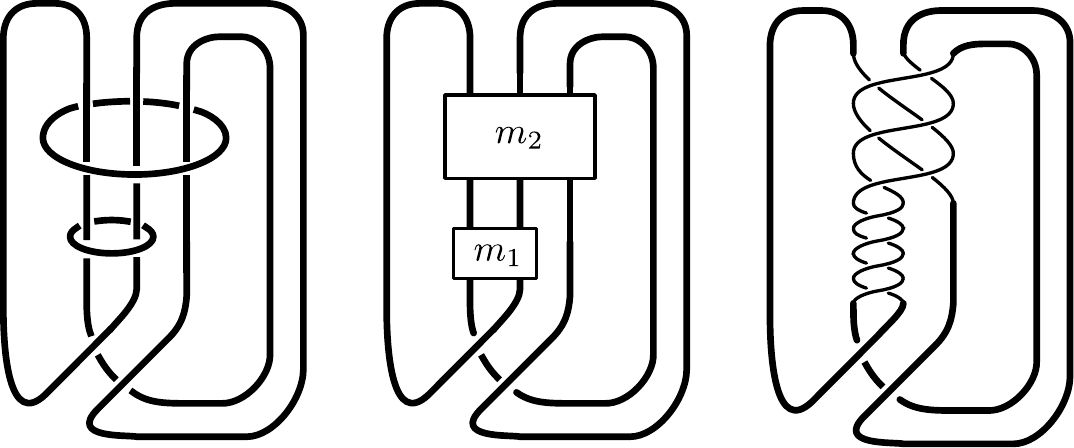}
\caption{Left: The seed link $K$ and the 2-fusion knot $K(m_1,m_2)$.
As an example $K(2,1)$ is the $(-2,3,7)$ pretzel knot.}
\label{f.seedlink}
\end{center}
\end{figure}

The 2-parameter family of 2-fusion knots includes the 2-strand torus knots,
the $(-2,3,p)$ pretzel knots and some knots that appear in the work of
Gordon-Wu related to exceptional Dehn surgery~\cite{GW}. The non-Montesinos,
non-alternating knot $K(-1,3)=K4_3$ was the focus of~\cite{GL2} regarding
a numerical confirmation of the volume conjecture. 
The topology and geometry of 2-fusion knots is explained in detail in 
Section \ref{sub.2fusiontop}.

\subsection{Our results}
\label{sub.results}

Our main Theorem~\ref{thm.1} gives an explicit formula for the
Jones slope for all 2-fusion knots $K(m_1,m_2)$. 
Recall that the Jones slope(s) $js_K$ of a knot $K$ is the set of values of
the periodic function $c_{K,2}:\BN\to \BQ$ that governs the leading order
of the $q$-degree of $J_{K,n}(q)$. In our case set of Jones slopes is a 
singleton for each pair $m_1,m_2$ so we denote by $\js(m_1,m_2)\in \BQ$ 
the unique element of the set of Jones slopes of $K(m_1,m_2)$.
The formula for $\js$ is a piece-wise rational function of $m_1,m_2$ defined 
on the lattice points $\BZ^2$ of the plane, which are partitioned into five 
sectors shown in color-coded fashion in Figure~\ref{fig.Slopes}. 
The reader may observe that the 5 branches of the function $\js: \BZ^2\to\BQ$ 
do not agree when extrapolated. 
For example for $m_1<1$ and $m_2 = 0$ the formula $2m_2+\frac{1}{2}$ from the 
red region does not agree (when extrapolated) with the actual value $0$ for 
the Jones slope at $m_2=0$. 
This disagreement disappears when we study the corresponding real optimization 
problem in Section~\ref{sec.rl} below. The branches given there
actually fit together continuously.


\begin{figure}[!htpb]
\begin{center}
\includegraphics[height=0.45\textheight]{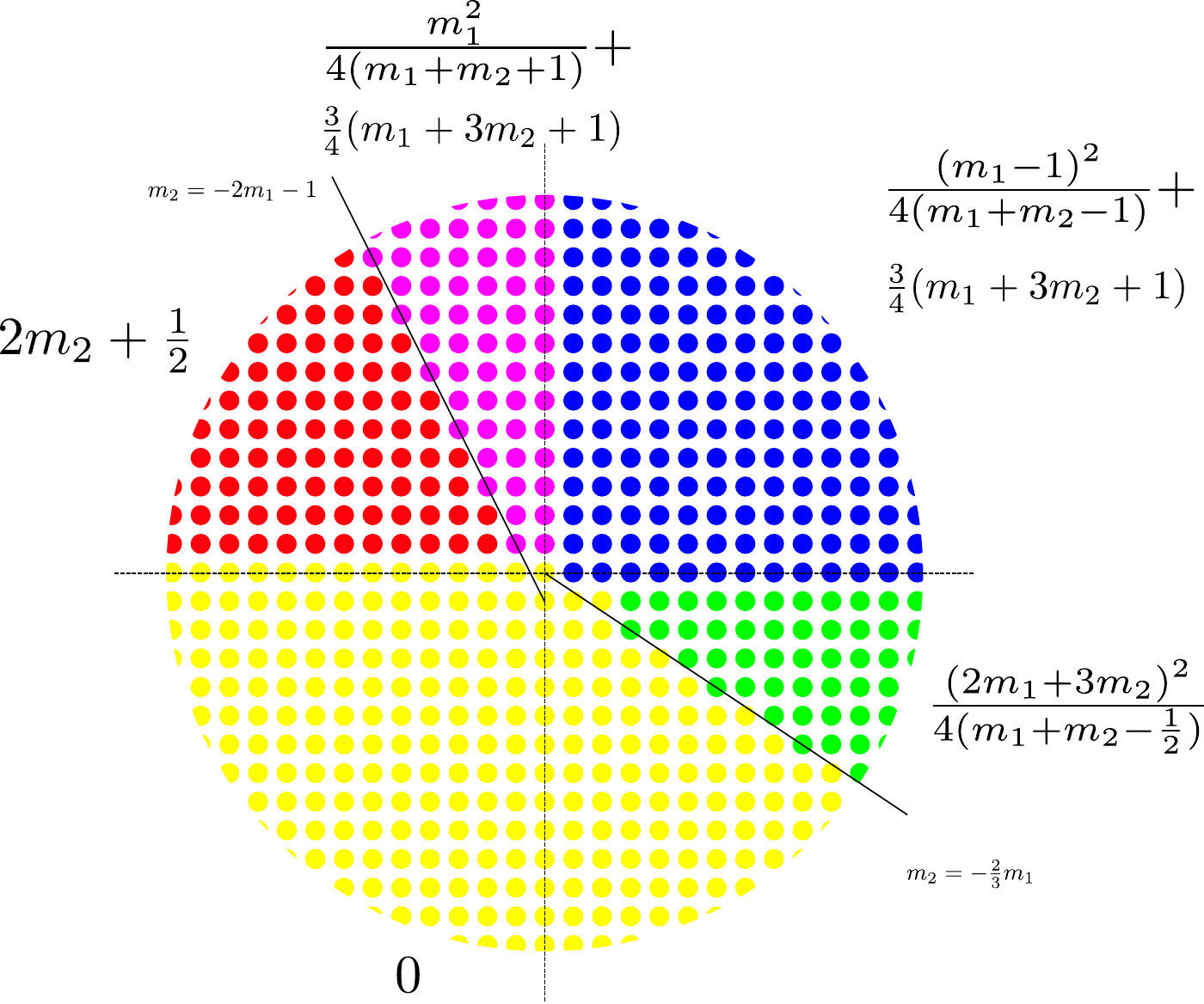}
\caption{The formula for the Jones slope of $K(m_1,m_2)$.}
\label{fig.Slopes}
\end{center}
\end{figure}

\begin{theorem}
\label{thm.1}
For any $m_1,m_2$ there is only one Jones slope. 
Moreover, if we divide the $(m_1,m_2)$-plane into regions as shown in Figure 
\ref{fig.Slopes} then the Jones slope $\js(m_1,m_2)$ of $K(m_1,m_2)$ is 
given by:
\begin{equation}
\label{eq.jslope}
\js(m_1,m_2)=
\begin{cases}
\frac{(m_1-1)^2}{4(m_1+m_2-1)}
+\frac{3m_1+9 m_2+3}{4} 
& \text{if} \,\, m_1 \geq 1, \,\, m_2\geq 0 \\
\frac{m_1^2}{4 (m_1+m_2+1)}
+\frac{3m_1+9 m_2+3}{4} 
& \text{if} \,\, m_1\leq 0, \,\, m_2\geq -1-2m_1, \,\, m_2 \geq 1 \\
2m_2 +\frac{1}{2}
& \text{if} \,\, 0< m_2, \,\, m_2< -1-2m_1 \\
0  
& \text{if} \,\, m_2\leq 0, \,\, m_2 \leq -\frac{2}{3}m_1,
\,\, \text{or} \,\, (m_1,m_2)=(2,-1) \\
\frac{(2m_1 + 3 m_2)^2}{4 (m_1+m_2-\frac{1}{2})}
& \text{if} \,\, m_2 > -\frac{2}{3}m_1, \,\, m_2\leq -1
\end{cases}
\end{equation}
with $\js(1,0)=3/2$.
\end{theorem}

Combining the work of \cite[Thm.1.9]{DG} we obtain a proof for the 
slope conjecture for a large class of $2$-fusion knots. 

\begin{corollary}
\label{cor.slopeconj}
The slope conjecture is true for all 2-fusion knots $K(m_1,m_2)$ with 
$m_1>1,m_2>0$.
\end{corollary}

As the knots are generally non-Montesinos this result is beyond the reach
of other known techniques. Also the Jones slopes are of great interest in 
that they are generally not integers so that they can not be found using 
semi-adequacy.

We should remark that the incompressibility criterion of \cite{DG} can 
also be applied to prove the slope conjecture for the remaining $2$-fusion 
knots. However, this is not the focus of the present paper, and we will not
provide any further details on this separate matter.

\begin{remark}
\label{rem.involution}
Using the involution 
\begin{equation}
\label{eq.involution}
K(m_1,m_2)=-K(1-m_1,-1-m_2), \qquad
K(-1,m_2)=K(-1,-m_2)
\end{equation} 
Theorem~\ref{thm.1} computes the Jones slopes of the mirror of the
family of 2-fusion knots. Hence, for every 2-fusion knot, we obtain two
Jones slopes.
\end{remark}

\begin{remark}
\label{rem.degree}
The proof of Theorem \ref{thm.1} also gives a formula for the degree
of the colored Jones polynomial. This formula is valid for all $n$, and it
is manifestly a quadratic quasi-polynomial. See Section \ref{sec.rl}.
\end{remark}

\begin{remark}
\label{rem.latticereal}
Theorem~\ref{thm.1} has a companion Theorem~\ref{thm.3} which is the 
solution to a real quadratic optimization problem. Theorem~\ref{thm.2}
implies the existence of a function $\jsr:\BR^2\to\BR$ with the following
properties:
\begin{itemize}
\item[(a)] 
$\jsr$ is continuous and piece-wise rational, with corner
locus (i.e., locus of points where $\jsr$ is not differentiable) given by 
quadratic equalities and inequalities whose complement
divides the plane $\BR^2$ into 9 sectors, shown in Figure~\ref{fig.realmax}.
\item[(b)] $\jsr$ is a real interpolation of $\js$ in the sense that it
satisfies
$$
\jsr(m_1,m_2)=\js(m_1,m_2)
$$
for all integers $m_1,m_2$  except those of the form $(m_1,0)$ with 
$m_1\leq 0$ and $(2,-1)$. See Corollary~\ref{cor.jj} below.
\item[(c)] Each of the 9 branches of $\jsr$ (after multiplication by 4)
becomes a boundary slope of $K(m_1,m_2)$ valid in the corresponding
region, detected by the incompressibility
criterion of~\cite[Sec.8]{DG}.
\end{itemize}
\end{remark}


\section{The colored Jones polynomial of 2-fusion knots}
\label{sec.cj}

\subsection{A state sum for the colored Jones polynomial}
\label{sub.state}

The cut-and-paste axioms of TQFT allow to compute the quantum invariants
of knotted objects in terms of a few building blocks, using a combinatorial
presentation of the knotted objects. In our case, we are interested in
state sum formulas for the colored Jones function $J_{K,n}(q)$ of a knot
$K$. Of the several state sum formulas available in the literature, 
we will use the {\em fusion formulas} that appear in \cite{CFS,Co,MV,GV,KL,Tu}. 
Fusion of knots are knotted trivalent graphs. There are five building blocks 
of fusion (the functions $\mu,\nu,\U,\Theta,\Tet$ below), expressed in 
terms of quantum factorials. Recall the {\em quantum
integer} $[n]$ and the {\em quantum factorial} $[n]!$ of a natural number 
$n$ are defined by
$$
[n]=\frac{q^{n/2}-q^{-n/2}}{q^{1/2}-q^{-1/2}},
\qquad
[n]!=\prod_{k=1}^n [k]!
$$
with the convention that $[0]!=1$. Let 
$$
\qbinom{a}{a_1, a_2, \dots, a_r}=\frac{[a]!}{[a_1]! \dots [a_r]!}
$$
denote the multinomial coefficient of natural numbers $a_i$ such that
$a_1+ \dots a_r=a$. We say that a triple $(a,b,c)$ of natural
numbers is {\em admissible} if $a+b+c$ is even and the triangle inequalities
hold. In the formulas below, we use the following basic trivalent graphs
$\U,\Theta,\Tet$ colored by one, three and six natural numbers (one
in each edge of the corresponding graph) such that
the colors at every vertex form an admissible triple shown in 
Figure~\ref{f.3j6j}.

\begin{figure}[!htpb]
\begin{center}
\includegraphics[height=0.10\textheight]{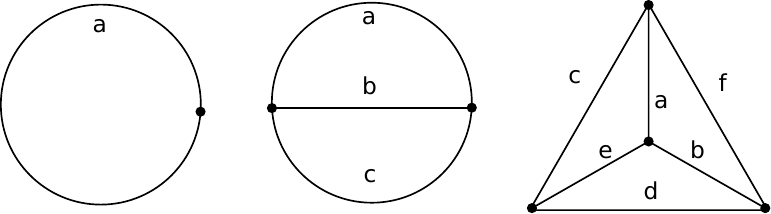}
\label{f.3j6j}
\end{center}
\end{figure}

Let us define the following functions.

\begin{eqnarray*}
\mu(a)&=& (-1)^a q^{\frac{-a(a+2)}{4}} \\
\nu(c,a,b)&=& (-1)^{\frac{a+b-c}{2}} q^{\frac{a(a+2)+b(b+2)-c(c+2)}{8}} \\
\U(a)&=&(-1)^a[a+1] \\
\Theta(a,b,c)&=& (-1)^{\frac{a+b+c}{2}}[\frac{a+b+c}{2}+1]
\qbinom{\frac{a+b+c}{2}}{\frac{-a+b+c}{2}, \frac{a-b+c}{2}, \frac{a+b-c}{2}}
\\
\Tet(a,b,c,d,e,f)&=& 
\sum_{k = \max T_i}^{\min S_j} (-1)^k [k+1]
\qbinom{k}{S_1-k , S_2-k , S_3-k , k- T_1 , k- T_2 , k- T_3 , k- T_4}
\end{eqnarray*}
where
\begin{equation}
\label{eq.Sj}
S_1 = \frac{1}{2}(a+d+b+c)\qquad S_2 = \frac{1}{2}(a+d+e+f) 
\qquad S_3 = \frac{1}{2}(b+c+e+f)
\end{equation}
\begin{equation}
\label{eq.Ti}
T_1 = \frac{1}{2}(a+b+e) \qquad T_2 = \frac{1}{2}(a+c+f)
\qquad T_3 = \frac{1}{2}(c+d+e) \qquad T_4 = \frac{1}{2}(b+d+f).
\end{equation}
An assembly of the five building blocks can compute the colored Jones
function of any knot. The next theorem is an exercise in fusion
following word for word the proof of \cite[Thm.1]{GL2}. An elementary
and self-contained introduction to fusion is available in \cite[Sec.3.2]{GL2}.
In particular, the calculation of the colored Jones polynomial of 
the $2$-fusion knot $K(-1,3)$ (generalized verbatim to all $2$-fusion knots)
is given in \cite[Sec.3.3, p.390]{GL2}.

Consider the function
\begin{eqnarray}
\label{eq.Smnk}
S(m_1,m_2,n_,k_1,k_2)(q) &=&
\frac{\mu(n)^{-w(m_1,m_2)}}{\U(n)} 
\nu(2k_1,n,n)^{2 m_1+2 m_2} \nu(n+2 k_2,2k_1,n)^{2m_2+1} 
\\ & & \notag
\cdot \frac{\U(2k_1)\U(n+2k_2)}{\Theta(n,n,2k_1) \Theta(n,2k_1,n+2k_2)}
\Tet(n,2k_1,2k_1,n,n,n+2k_2) \,.
\end{eqnarray}

\begin{theorem}
\label{thm.cjk}
For every $m_1,m_2 \in \BZ$ and $n \in \BN$, we have:
\begin{equation}
\label{eq.cjk}
J_{K(m_1,m_2),n}(q) = \sum_{(k_1,k_2) \in nP \cap \BZ^2} 
S(m_1,m_2,n_,k_1,k_2)(q) \,,
\end{equation}
where $P$ is the polytope from Figure~\ref{fig.P} and 
the {\em writhe} of $K(m_1,m_2)$ is given by $w(m_1,m_2)=2m_1+6m_2+2$.
\end{theorem}


\begin{figure}[!htpb]
\begin{center}
\includegraphics[height=0.15\textheight]{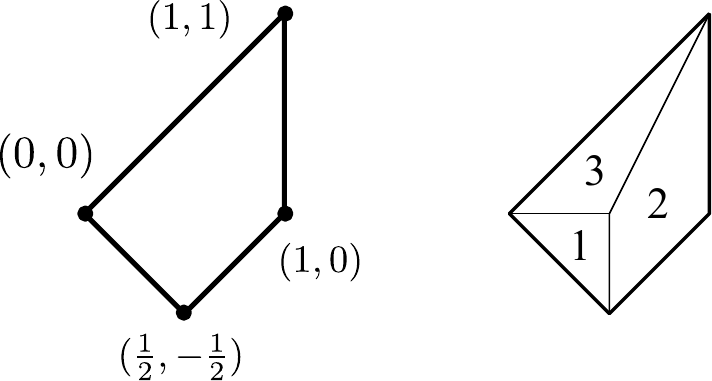}
\caption{The polygon $P$ on the left and its decomposition into three 
regions $P_1,P_2,P_3$
on the right.}
\label{fig.P}
\end{center}
\end{figure}


\begin{remark}
\label{rem.P}
Notice that for every $n \in \BN$, we have:
$$
\{(k_1,k_2) \in \BZ^2 \,\, | 0 \leq 2 k_1 \leq 2n, \quad |n-2k_1| 
\leq n+2 k_2 \leq n+2k_1\}=
nP \cap \BZ^2.
$$
For the purpose of visualization, we show the lattice points in $4P$ and $5P$
in Figure~\ref{f.P45}.
\end{remark}

\begin{figure}[!htpb]
\begin{center}
\includegraphics[height=0.15\textheight]{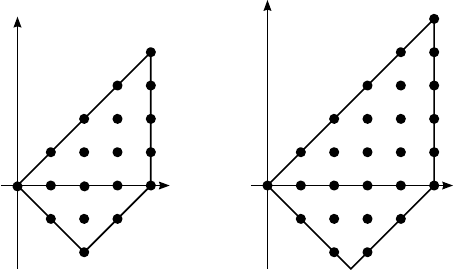}
\caption{The lattice points in $4P$ and $5P$.}
\label{f.P45}
\end{center}
\end{figure}

\subsection{The leading term of the building blocks}
\label{sub.building}

In this section we compute the leading term of the five building blocks of 
our state sum.

\begin{definition}
\label{def.tropical}
If $f(q) \in  \BQ(q^{1/4})$ is a rational function, let $\d(f)$ and
 $\lt(f)$ the minimal {\em degree} and the {\em leading coefficient} 
of the Laurent expansion of $f(q) \in \BQ((q^{1/4}))$ with respect to 
$q^{1/4}$. Let 
\begin{equation}
\label{eq.hatf}
\widehat{f}(q)=\lt(f) q^{\d(f)}
\end{equation}
denote the leading term of $f(q)$. 
\newline
\end{definition}
We may call $\widehat f(q)$ the {\em tropicalization} of $f(q)$. 
Observe the trivial but useful identity:
\begin{equation}
\label{eq.ltprod}
\widehat{f g}=\hat f \, \hat g
\end{equation}
for nonzero functions $f,g$.

\begin{lemma}
\label{lem.tropicalblocks}
For all admissible colorings we have:
\begin{eqnarray*}
\lt(\mu)(a)&=& (-1)^a \\
\lt(\nu)(c,a,b)&=& (-1)^{\frac{a+b-c}{2}} \\
\lt(\U)(a)&=&(-1)^a  \\
\lt(\Theta)(a,b,c)&=& (-1)^{\frac{a+b+c}{2}}
\\
\lt(\Tet)(a,b,c,d,e,f)&=& (-1)^{ k^*}
\end{eqnarray*}
where
$$
k^*=\min S_j
$$
and
\begin{eqnarray*}
\d(\mu)(a)&=&  \frac{-a(a+2)}{4} \\
\d(\nu)(c,a,b)&=&  \frac{a(a+2)+b(b+2)-c(c+2)}{8} \\
\d(\U)(a)&=& \frac{a}{2} \\
\d(\Theta)(a,b,c)&=& -\frac{1}{8}(a^2+b^2+c^2)+\frac{1}{4}(ab+ac+bc)
+\frac{1}{4}(a+b+c)
\\
\d(\Tet)(a,b,c,d,e,f)&=& \d(b_7)(S_1-k^* , S_2-k^* , S_3-k^* , 
k^*- T_1 , k^*- T_2 , k^*- T_3 , k^*- T_4) + \frac{k^*}{2}
\end{eqnarray*}
where $S_j$ and $T_i$ are given in Equations \eqref{eq.Sj} and \eqref{eq.Ti},
$b_7(a_1,\dots,a_7)=\qbinom{a}{a_1, a_2, \dots, a_7}$ is the 7-binomial
coefficient and
\begin{eqnarray*} 
\d(b_7)(a_1,\dots,a_7)&=&\frac{1}{4} \left(\left(\sum_{i=1}^7 a_i\right)^2
-\sum_{i=1}^7 a_i^2\right).
\end{eqnarray*}
\end{lemma}

\begin{proof}
Use the fact that
$$
\widehat{[a]}=q^{\frac{a-1}{2}}
$$ 
and
$$
\widehat{[a]!}=q^{\frac{a^2-a}{4}}
$$
This computes the leading term of $\Theta$ and of the quantum multinomial
coefficients. Now $\Tet(a,b,c,d,e,f)$ is given by a 1-dimensional sum
of a variable $k$. It is easy to see that the leading term comes the maximum
value $k^*$ of $k$. The result follows.
\end{proof}


\subsection{The leading term of the summand}
\label{sub.tropical}

Consider the function $Q$ defined by

\begin{eqnarray}
\label{eq.Q}
Q(m_1,m_2,n,k_1,k_2)&=&
\frac{k_1}{2} - \frac{3 k_1^2}{2} - 3 k_1 k_2 - k_2^2 
- k_1 m_1 - k_1^2 m_1 - k_2 m_2 - 
 k_2^2 m_2 - 6 k_1 n 
\\ \notag 
& & 
- 3 k_2 n + 2 m_1 n + 4 m_2 n - k_2 m_2 n - 2 n^2 + 
 m_1 n^2 + 2 m_2 n^2
\\ \notag
& & +
\frac{1}{2} \left(
(1 + 8 k_1 + 4 k_2 + 8 n) \min\{l_1,l_2,l_3
\} - 3 \min\{l_1,l_2,l_3
\}^2\right)
\end{eqnarray}
where
$$
l_1=2 k_1 + n, \qquad l_2=2 k_1 + k_2 + n, \qquad l_3=k_2 + 2 n.
$$
Notice that for fixed $m_1,m_2$ and $n$, the function 
$k=(k_1,k_2) \mapsto Q(m_1,m_2,n,k)$ is piece-wise quadratic function.
Moreover, for all $m_1, m_2$ and $n$ the restriction of the above 
function to each region of $nP$ is a quadratic function of $(k_1,k_2)$.   

\begin{lemma}
\label{lem.ltQ}
For all $(m_1,m_2,n_,k_1,k_2)$ admissible, we have
$$
\hat S(m_1,m_2,n_,k_1,k_2) = (-1)^{k_1+n+\min\{2k_1,2k_1+k_2,k_2+n\}}
q^{Q(m_1,m_2,n_,k_1,k_2)}
$$
\end{lemma}

\begin{proof}
It follows easily from Section \ref{sub.building} and 
Equation~\eqref{eq.ltprod}.
\end{proof}


\section{Proof of Theorem \ref{thm.1}}
\label{sec.thm1}

The proof involves four cases:

\begin{center}
\begin{tabular}{|c|c|c|c|} \hline
Case 0 & Case 1 & Case 2 & Case 3 \\ \hline
$m_2\in \{0,-1\}$ & $m_1,m_2\geq 1$ & $m_1\leq 0,m_2\geq 1$ & $m_2\leq -2$
\\ \hline
\end{tabular}
\end{center}

Case $0$ involves only alternating torus knots since $K(m_1,0) = T(2,2m_1+1)$ 
and $K(m_1,-1) = T(2,2m_1-3)$ for which the Jones slopes were already 
known~\cite{Ga1}.

In the remaining three cases we will take the following steps:

\begin{enumerate}
\item 
Estimate partial derivatives of $Q$ in the various regions $P_i$ to narrow 
down the location of the lattice points that achieve the maximum of $Q$ on 
$nP\cap \BZ^2$. In all cases they will be on a single boundary line of $Q$.
\item 
Since the restriction of $Q$ to a boundary line is an explicit quadratic 
function in one variable, there can be at most $2$ maximizers and we can 
readily compute them.
\item 
If there are two maximizers, compute the leading term of the corresponding 
summand to see if they cancel out.
\item 
If there is no cancellation, then we can evaluate $Q(m_1,m_2,n,k)/n^2$ at 
either of the maximizers $k$ to get the slope.
\item 
If there is cancellation we first have to explicitly take together all 
the canceling terms until no more cancellation occurs at the top degree. 
This happens only in the difficult Case 3.
\end{enumerate}

\subsection{Case 1: $m_1,m_2\geq 1$}
\label{sub.case1}

Recall that $Q_i$ is $Q$ restricted to the region $nP_i$ defined in Figure \ref{fig.P}. We have:
\begin{align}
\frac{\partial Q_1}{\partial k_2}<0 \qquad 
\frac{\partial Q_2}{\partial k_1},\frac{\partial Q_2}{\partial k_2}<0
\qquad 
\frac{\partial Q_3}{\partial k_2}<0 \,.
\end{align}

Before we may conclude that the maximum of $Q$ on $nP\cap\BZ^2$ is on the 
line $k_2=-k_1$ we have to check the following. For odd $n$ there could be 
a \emph{jump} across the line $k=\frac{n}{2}$ between
regions $nP_2$ and $nP_1$. We therefore set $n = 2N+1$ explicitly check that
\[Q_1(m_1, m_2, 2 N + 1, N, -N) - Q_2(m_1, m_2, 2 N + 1, N + 1, -N)>0
\] 
Restricted to the line $k_2=-k_1$, $Q$ is a negative definite quadratic in 
$k_1$ with critical point
\[
c_1 = \frac{1 - m_1 + m_2 + m_2 n}{2 (-1 + m_1 + m_2)} 
\] 
For $m_1 >1$ we have $c_1\in (-\frac{1}{2},\frac{n}{2}]$ and
for $m_1 = 1$ we have $c_1= \frac{n+1}{2}$. In both cases the maximizers are
the lattice points in the diagonal closest to $c_1$ satisfying 
$k_1\leq \frac{n}{2}$.
There may be a tie for the maximum between two adjacent points. To rule out 
the possibility of cancellation we take a look at the leading term restricted 
to the line $k_2=-k_1$. The leading term is
$(-1)^n$. Since the sign of the leading term is independent of $k_1$ along the 
diagonal, there cannot be cancellation. We may conclude that the slope is 
given by the constant term of $Q(m_1,m_2,n,c_1,-c_1)/n^2$.
This gives the slope 
$\frac{(m_1-1)^2}{4(m_1+m_2-1)}
+\frac{m_1+9 m_2+1}{4}$ 
indicated in the blue region of Figure \ref{fig.Slopes}.

\subsection{Case 2: $m_1\leq 0,m_2\geq 1$}
\label{sub.case2}

We have:
\begin{align}
\frac{\partial Q_1}{\partial k_1}>0,\frac{\partial Q_1}{\partial k_2}<0
\qquad 
\frac{\partial Q_2}{\partial k_2}<0 \qquad
\frac{\partial Q_3}{\partial k_2}<0 \,.
\end{align}

Before we may conclude that the maximum of $Q$ on $nP\cap\BZ^2$ is on the 
line $k_2=k_1-n$ we have to check the following.
For odd $n$ there could be a \emph{jump} across the line $k_1=\frac{n}{2}$ 
between regions $nP_2$ and $nP_1$. We therefore set $n = 2N+1$ explicitly 
check that
\[Q_2(m_1, m_2, 2 N + 1, N+1, -N) - Q_1(m_1, m_2, 2 N + 1, N, -N)>0
\] 

Restricted to the line $k_2=k_1-n$ the coefficient of $k_1^2$ in $Q$ is 
$a = -1 - m_1 - m_2$.
If $a>0$ the critical point $c_2$ is given by
\[
c_2 = \frac{1 - m_1 + m_2 + m_2 n}{2 (-1 + m_1 + m_2)} 
\] 
Since $c_2<\frac{3}{4}n$ the maximizer is given by $k_1=n$ and so the slope is:
$2m_2+\frac{1}{2}$ as shown in red in Figure \ref{fig.Slopes}.
If $a=0$ we have the same conclusion because along the diagonal $Q$ is now 
an increasing linear function in $k_1$. Finally if $a<0$ we need to 
determine if $c_2\in [\frac{n}{2}-\frac{1}{2},n+\frac{1}{2}]$.

We always have $c_2>\frac{n-1}{2}$, and if in addition $1 + 2 m_1 + m_2<0$ 
then $c_2>n-1/2$. This means the maximizer is $k_1=n$ and the slope is 
$\frac{1}{2}+2m_2$ as shown in red in Figure \ref{fig.Slopes}.  

If $1 + 2 m_1 + m_2 \geq 0$ then $c_2\in [\frac{n-1}{2},n+\frac{1}{2}]$ 
and the maximizers are the lattice points on the line closest to $c_2$. 
There may of course be cancellation if there is a tie. To rule this out we 
check that along the line the sign of the leading term is independent of 
$k_1$. Indeed the leading term on this line is $(-1)^n$.

We may conclude that the slope is given by the constant term of 
$Q(m_1,m_2,n,c_2,c_2-n)/n^2$
This gives the slope 
$\frac{m_1^2}{4 (m_1+m_2+1)}
+\frac{m_1+9 m_2+1}{4} $ 
indicated in the purple region of Figure \ref{fig.Slopes}.

\subsection{Case 3: $m_1\leq 0,m_2\leq -2$}
\label{sub.case3}

One can check that:
\begin{align}
\frac{\partial Q_1}{\partial k_2}>0 \qquad
\frac{\partial Q_2}{\partial k_2}>0 \qquad
\frac{\partial Q_3}{\partial k_2}>0 \,.
\end{align}
This means that the lattice maximizers of $Q$ will be on the diagonal 
$k_1=k_2$. Here the restriction of $Q$ is a quadratic and the coefficient 
of $k_1^2$ is $\frac{1}{2}-m_1-m_2$. If $m_1\leq -m_2$ then it is positive 
definite with critical point given by
\[
c_3=\frac{-3 + 2 m_1 + 2 m_2 + 2 n + 2 m_2 n}{2 (1 - 2 m_1 - 2 m_2)}
\]
We have $c_3<0$ so the maximum is attained at $k_1=n$ giving a slope
of $0$ as shown in yellow in Figure \ref{fig.Slopes}.

If $m_1> -m_2$ the quadratic $Q$ is negative definite on the diagonal 
and the critical point $c_3$ satisfies $c_3>-\frac{1}{2}$. Furthermore 
$c_3 \geq n-\frac{1}{2}$ if and only if 
$- 3 m_2 \geq 2 m_1$ and this case we get again 
the maximizer $k_1=n$ and slope $0$.

The only remaining case is $2m_1 > -3 m_2$, which means 
$c_3\in (-\frac{1}{2},n-\frac{1}{2}]$. Here we have to check for cancellation
and indeed, there will be cancellation along a subsequence since the 
leading term alternates along the diagonal, it is $(-1)^{k_1+n}$.

To finish the proof we must rule out the possibility of a new slope 
occurring when the degree drops dramatically due to cancellation. Below we 
will deal with the cancellation and show the drop in degree is at most 
linear in $n$ so that no new slope can appear. Our conclusion then is that 
the slope is given by the constant term of $Q_3(m_1,m_2,n,c_3,c_3)/n^2$
which is:
$\frac{(2m_1 + 3 m_2)^2}{4 (m_1+m_2-\frac{1}{2})}$
as shown in green in Figure \ref{fig.Slopes}.

\subsection{Analysis of the cancellation in Case 3}
\label{sub.cancel}

Cancellation happens exactly when the critical point on the diagonal is 
a half integer $c_3 \in \frac{1}{2}+\BZ$. Note also that not just the two 
terms tying for the maximum cancel out. All the terms along the diagonal 
corresponding to $k_1=c_3\pm\frac{2b+1}{2}$ cancel out to some extent. Here 
$b = 0\hdots \min(c_3,n-c_3)-\frac{1}{2}$.

Along the diagonal the Tet consists of a single term so that the summand $S$ 
simplifies considerably, call it $D$:
\begin{align*}
D(k):= S(m_1,m_2,n,k,k) &
= (-1)^{(2 m_2 + 1) n/2 + n} q^{-(2 m_2 + 1) n^2/8} [n]!\times\\
& \quad (-1)^{k} q^{-(m_1 + m_2) k (k + 1) - (2 m_2 + 1) n (2 k + 1)/4} 
\frac{[n + 2 k + 1] [2 k + 1]!}{[k]! [n + k + 1]!}
\end{align*}

To see how far the degree drops when taking together the canceling terms in 
pairs and take together $D(k)$ and $D(k-a)$. For $a\in \BN$ the result is:

\begin{align*}
D(k)+D(k-a) &= 
C\left(q^{\alpha} \{n + 2 k - 2 a + 1\} 
\frac{\{k\}! \{n + k + 1\}!}{\{k - a\}! \{n + k - a + 1\}!} \right. \\
& \quad + \left. 
 (-1)^s q^{\beta} \{n + 2 k + 1\} 
\frac{\{2 k + 1\}!}{\{2 k - 2 a + 1\}!}\right) \,.
\end{align*}

Here $C$ is an irrelevant common factor and in case of cancellation the 
monomials $q^\alpha$ and $(1)^sq^\beta$ are determined to make the leading 
terms of equal degree and opposite sign. Lastly we have taken out all 
denominators of the quantum numbers and factorials and define 
$\{k\} = [k](q^{\frac{1}{2}}-q^{-\frac{1}{2}})$.  

Since we assume the leading terms cancel we investigate the next degree term 
in both parts of the above formula. For this we can ignore $C$ and the 
monomials and restrict ourselves to the two products of terms of the form 
$\{x\}$. Both products can be simplified to remove the denominator. The 
difference in degree between the two terms of $\{x\}$ is exactly $x$. If 
$\{x\}$ is the least integer that occurs in the product then the difference 
in degree between the leading term and the highest subleading term is 
exactly $x$. For the first term $x$ is $k-a+1$ and for the second term it is
$x = 2k-2a+2$. In conclusion the highest subleading term does not cancel out 
and has degree exactly $k-a+1$ lower than the leading term.

To finish the argument we would like to show that the $b=0$ terms 
$k_1=c_3\pm\frac{1}{2}$ still produce the highest degree term after 
cancellation. This is not obvious since the degree drops by exactly 
$c_3-b+\frac{1}{2}$. In other words after cancellation
the degree of the terms corresponding to $b$ gains exactly $b$ relative to 
the $b=0$ terms. To settle this matter we show that the
difference in degree before cancellation was more than $b$.
\[
Q_3(m_1,m_2,n,c_3+\frac{1}{2},c_3
+\frac{1}{2})-Q_3(m_1,m_2,n,c_3-b-\frac{1}{2},c_3-b-\frac{1}{2}) = 
\frac{b (1 + b)}{2} (-1 + 2 m_1 + 2 m_2) > b
\]
Because $b\geq1$ and $2m_1>-3m_2$ so $-1+2m_1+2m_2>-1-m_2\geq 1$.

The same computation also shows how to deal with the diagonal terms where 
$b>\min(c_3,n-c_3)-\frac{1}{2}$ that did not suffer any cancellation
because their symmetric partner was outside of $nP$. We need to show that 
the difference in degree before cancellation is at least
$c_3+\frac{1}{2}$. So for $b=\min(c_3,n-c_3)-\frac{1}{2}$ check explicitly that
$\frac{b (1 + b)}{2} (-1 + 2 m_1 + 2 m_2)>c_3+\frac{1}{2}$. This is true 
provided that $n>m_1$. 

Finally we check that the degree of the $b=0$ terms before cancellation is 
greater than $c_3+\frac{1}{2}$ plus the degree of any off-diagonal term. 
For this we only need to consider the terms $(k_1,k_2) = (k_1,k_1-1)$. 
Again it follows by a routine computation.


\section{Real versus lattice quadratic optimization}
\label{sec.rl}

\subsection{Real quadratic optimization with parameters}
\label{sub.realop}

In this section we study the real quadratic optimization problem of 
Equation~\eqref{eq.drkn} and compare it with the lattice quadratic
optimization problem of Theorem~\ref{thm.1}.

Fix a rational convex polytope $P$ in $\BR^r$ and a piece-wise quadratic 
function $\delta$ in the variables $n,x$ where $x=(x_1,\dots,x_r)$.
Then, we have:
$$
\hat \d_{\BR}(n):=\max\{ \d(n,x) \,\, | \,\, x \in n P \} =
\max\{ \d(n,nx) \,\, | \,\, x \in P \}  \,.
$$
Observe that $\d(n,nx)$ is a quadratic polynomial in $n$ with coefficients
piece-wise quadratic polynomial in $x$. it follows that for $n$ large
enough, $\hat \d_{\BR}(n)$ is given by a quadratic polynomial in $n$.
If $\jsr$ denote the coefficient of $n^2$ in $\hat \d_{\BR}(n)$, 
and $\d_2(x)$ denotes the coefficient of $n^2$ in $\d(n,nx)$
then we have:
$$
\jsr = \max\{ \d_2(x) \,\, | \,\, x \in P \}  \,.
$$
If $\d$ depends on some additional parameters $m \in \BR^r$, then
we get a function 
\begin{equation}
\label{eq.jsr}
\jsr: \BR^r \mapsto \BR \,.
\end{equation} 
Assume that dependence of $\d$ on $m$ is polynomial with real coefficients.
To compute $\jsr(m)$, consider the piece-wise quadratic polynomial
(in the $x$ variable) $\d_2(m,x)$, which achieves a maximum at some point of 
the compact set $P$. Subdividing $P$ if necessary, we may assume that 
$\d_2(m,x)$ is a polynomial in $x$. If the maximum $\hat x$ is at the interior 
of $P$, since $\d_2(m,x)$ is quadratic, its gradient is an affine linear
function of $x$, hence it has a unique zero. In that case, it follows that 
$\hat x$ is the
unique critical point of $\d_2(m,x)$ and $\d_2(m,x)$ has negative
definite quadratic part. Since the coefficients of the quadratic function 
$\d_2(m,x)$ of $x$ are polynomials in $m$, it follows that in the above case
the coefficients of $\hat x$ are rational functions of $m$. The condition 
that $\hat x$ is a maximum point in the interior of $P$ can be
expressed by polynomial equalities and inequalities on $m$. This defines
a {\em semi-algebraic set}~\cite{Basu}.
On the other hand, if $\hat x$ lies in the boundary of $P$, then either
$\hat x$ is a vertex of $P$ or there exists a face $F$ of $P$ such that
$\hat x$ lies in the relative interior of $F$. Restricting $\d_2(m,x)$
and using induction on $r$, or evaluating at $\hat x$ a vertex of $P$
implies the following.

\begin{theorem}
\label{thm.2}
With the above assumptions, $\jsr: \BR^r \mapsto \BR$ is a piece-wise
rational function of $m$, defined on finitely many sectors whose corner
locus is a closed semi-algebraic set of dimension at most $r-1$. 
Moreover, $\jsr$ is continuous.
\end{theorem}
Recall that the {\em corner locus} of a piece-wise function on $\BR^r$
is the set of points where the function is not differentiable.
Note that the proof  of Theorem~\ref{thm.2} is constructive, and easier 
than the corresponding lattice optimization problem, since we do not have to 
worry about ties. Moreover, since we are doing doing a sum, we
do not have to worry about cancellations.

\subsection{The case of $2$-fusion knots}
\label{sub.realop2}

We now illustrate Theorem~\ref{thm.2} for the case of $2$-fusion knots,
where $\d(m_1,m_2,n,x_1,x_2)$ is given by Equation~\eqref{eq.Q}. Notice that
$\d(m,n,x)$ is an affine linear function of $m=(m_1,m_2)\in \BR^2$.
A case analysis (similar but easier than the one of Section~\ref{sec.thm1}
shows the following.


\begin{figure}[!htpb]
\begin{center}
\includegraphics[height=0.30\textheight]{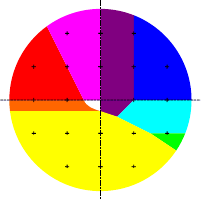}
\caption{The nine regions of $\jsr$ of Theorem~\ref{thm.3}.}
\label{fig.realmax}
\end{center}
\end{figure}

Define $\jsr(m_1,m_2)$ to be the real maximum of the summand for the fusion 
state sum of $K(m_1,m_2)$.

\begin{theorem}
\label{thm.3}
If we divide the $(m_1,m_2)$-plane into regions as shown in Figure 
\ref{fig.realmax} then $\jsr(m_1,m_2)$ is given by:
\[\jsr(m_1,m_2)=\]
\begin{equation}
\label{eq.jrs2}
\begin{cases}
\frac{(m_1-1)^2}{4(m_1+m_2-1)}
+\frac{3m_1+9 m_2+3}{4} 
& \text{if} \,\, m_1 > 1, \,\, m_2\geq 0 \\
\frac{3m_1+9m_2+3}{4}
& \text{if} \,\, 0 \leq m_1 \leq 1, \,\, 1 + m_1 + 3 m_2 \geq 0, 
\,\, 1 - m_1 + m_2 \geq 0\\
\frac{m_1^2}{4 (m_1+m_2+1)}
+\frac{3m_1+9 m_2+3}{4} 
& \text{if} \,\, m_1\leq 0, \,\, m_2 \geq 0, \,\, m_2\geq -1-2m_1 \\
2m_2 + \frac{1}{2}
& \text{if} \,\, m_2>0, \,\, 1 + 2 m_1 + m_2 \geq 0 \\
\frac{(3 m_2+1)^2}{4 (m_2+\frac{1}{2})}
& \text{if}\,\, -\frac{1}{3}\leq m_2 \leq 0, 
\,\, 1 + 2 m_1 + 3 m_2 + 4 m_1 m_2 \leq 0 \\
0  
& \text{if}\,\, m_2 \leq -\frac{1}{3}, \,\, 1 + m_1 + 3 m_2 \leq 0, 
\,\, 1 + 2 m_1 + 4 m_2 \leq 0, \,\, m_2 \leq -\frac{2}{3}m_1 \\
\frac{(2m_1 + 3 m_2)^2}{4 (m_1+m_2-\frac{1}{2})}
& \text{if} \,\, m_2 > -\frac{2}{3}m_1, \,\, m_2\leq -1\\
m_1 + 2 m_2 + \frac{1}{2} 
& \text{if} \,\, -1\leq m_2\leq 0 , \,\,  1 - m_1 + m_2 \leq 0, 
\,\, 1 + 2 m_1 + 4 m_2 \geq 0\\
I(m_1,m_2)
& \text{if} \,\, 1 + 2 m_1 + 3 m_2 + 4 m_1 m_2 \geq 0, 
\,\, -\frac{1}{2} \leq m_1 \leq 0, \,\, -\frac{1}{3} \leq m_2 \leq 0
\end{cases}
\end{equation}
where 
\[
I(m_1,m_2)
=\frac{3 + 6 m_1 + 4 m_1^2 + 18 m_2 + 24 m_1 m_2 + 8 m_1^2 m_2 + 27 m_2^2 + 
18 m_1 m_2^2}{4 (1 + m_1 + 3 m_2 + 2 m_1 m_2)}
\]
\end{theorem}

\begin{corollary}
\label{cor.jj}
An comparison between Theorems~\ref{thm.1} and \ref{thm.3} reveals that
$\js(m_1,m_2)=\jsr(m_1,m_2)$ for all pairs of integers $(m_1,m_2) \in \BZ^2$
except those of the form $(m_1,0)$ with $m_1\leq 0$ and $(2,-1)$.
For these exceptional pairs, $K(m_1,m_2)$ is a torus knot.
\end{corollary}


\section{$k$-seed links and $k$-fusion knots}
\label{sec.kfusion}

\subsection{Seeds and fusion}
\label{sub.defk}

There are several ways to tabulate and classify knots, and among them
\begin{itemize}
\item[(a)] 
by crossing number as was done by Rolfsen~\cite{Rf}, 
\item[(b)]
by the number of ideal tetrahedra (for hyperbolic knots) as is the standard 
in hyperbolic geometry~\cite{Th,snappy}, 
\item[(c)]
by arborescent planar projections, studied by Conway and 
Bonahon-Siebenmann~\cite{Co,BS},
\item[(d)]
by fusion~\cite{Thurston:shadow},
\item[(e)]
by shadows~\cite{Tu:shadow}.
\end{itemize}

Here we review the fusion construction of knots (and more generally, knotted 
trivalent graphs) which originates from cut and paste axioms in quantum 
topology. The construction was introduced by Bar-Natan and Thurston, appeared
in~\cite{Thurston:shadow} and further studied by the second author~\cite{V}.
Our definition of fusion is reminiscent to W. Thurston's hyperbolic Dehn 
filling \cite{Th}, and differs from a construction of knots by the 
same name (fusion) that appears in Kawauchi's book \cite[p.171]{Kawauchi}.


\begin{figure}[!htpb]
\begin{center}
\includegraphics[height=0.15\textheight]{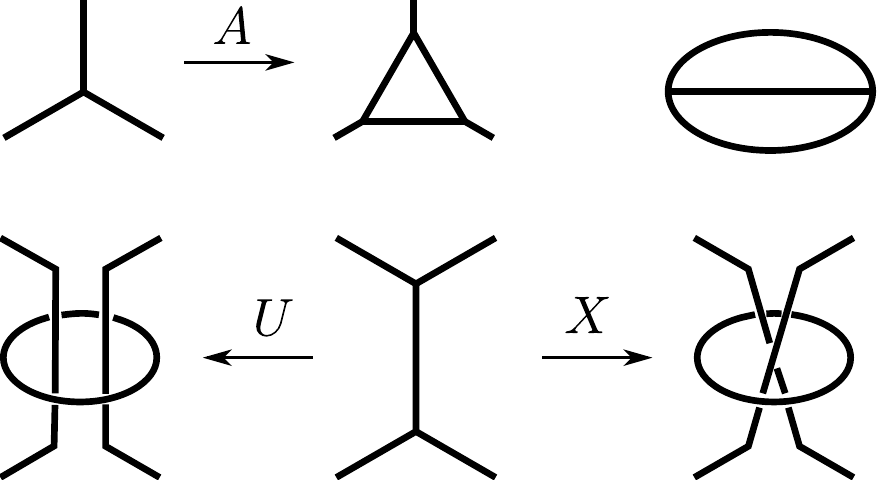}
\caption{The moves $A$,$U$,$X$ and the theta graph (upper right).}
\label{fig.Moves}
\end{center}
\end{figure}

\begin{definition}
\label{def.kseed}
A {\em seed link} is a link that can be produced from the theta graph by 
applying 
the moves $A,U,X$ shown in figure \ref{fig.Moves}. The additional components 
created by $U$ and $X$ are called \emph{belts}. 
A $k$-seed link is a seed link with $k$ belts.
\end{definition}

Note that the sign of the crossing introduced by the $X$-move is does not 
affect the complement of the seed link. If desired we may always perform all 
the $A$ moves first.

\begin{definition}
\label{def.kfusion}
Let $L$ be a $k$-seed link together with an ordering of its belts.
Define the $k$-fusion link $L(m_1,\hdots,m_k)$ to be the link obtained by
$-\frac{1}{m_j}$ Dehn filling on the $j$-th belt of $L$ for all $j=1,\dots,k$.
\end{definition}

Recall that the result of $-1/m$ {\em Dehn filling} along an unknot 
$C$ which 
bounds a disk $D$ replaces a string that meets $D$ with $m$ full twists, 
right-handed if $m>0$ and left-handed if $m<0$; see Figure~\ref{f.dehn} 
and~\cite{Kirby:calc}.


\begin{figure}[!htpb]
\begin{center}
\includegraphics[height=0.12\textheight]{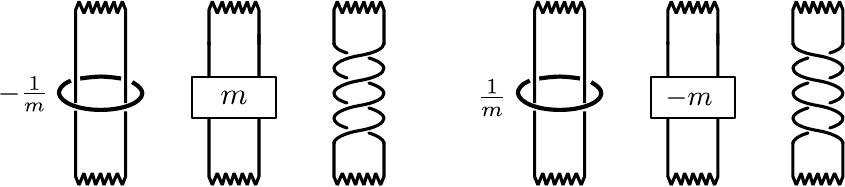}
\caption{The effect of Dehn filling on a link. In the picture we have taken 
$m=2$.}
\label{f.dehn}
\end{center}
\end{figure}


In a picture of a seed link the belts will always be enumerated from bottom 
to top. So for example the first belt of $K$ is the smallest one.

As suggested above, fusion is not just a way to produce a special class of 
knots. All knots and links can be presented this way although not in a 
unique way.

\begin{theorem}
\label{thm.fusion}
Any link is a $k$-fusion link for some $k$. 
The number of fusions is at most the number of twist regions of a diagram.
\end{theorem}

This theorem has its roots in Turaev's theory of shadows. A self-contained 
proof can be found in \cite{V}.

\subsection{$1$ and $2$-fusion knots}
\label{sub.12fusion}

We now specialize the discussion of $k$-fusion knots to the case $k=1,2$.
Figure~\ref{fig.Seeds} lists the sets of $1$-seed and $2$-seed links.
Since we are interested in knots, let $\calS_k$ denote the finite set of seed 
links with $k$ belts and $k+1$ components. 


\begin{figure}[!htpb]
\begin{center}
\includegraphics[height=0.30\textheight]{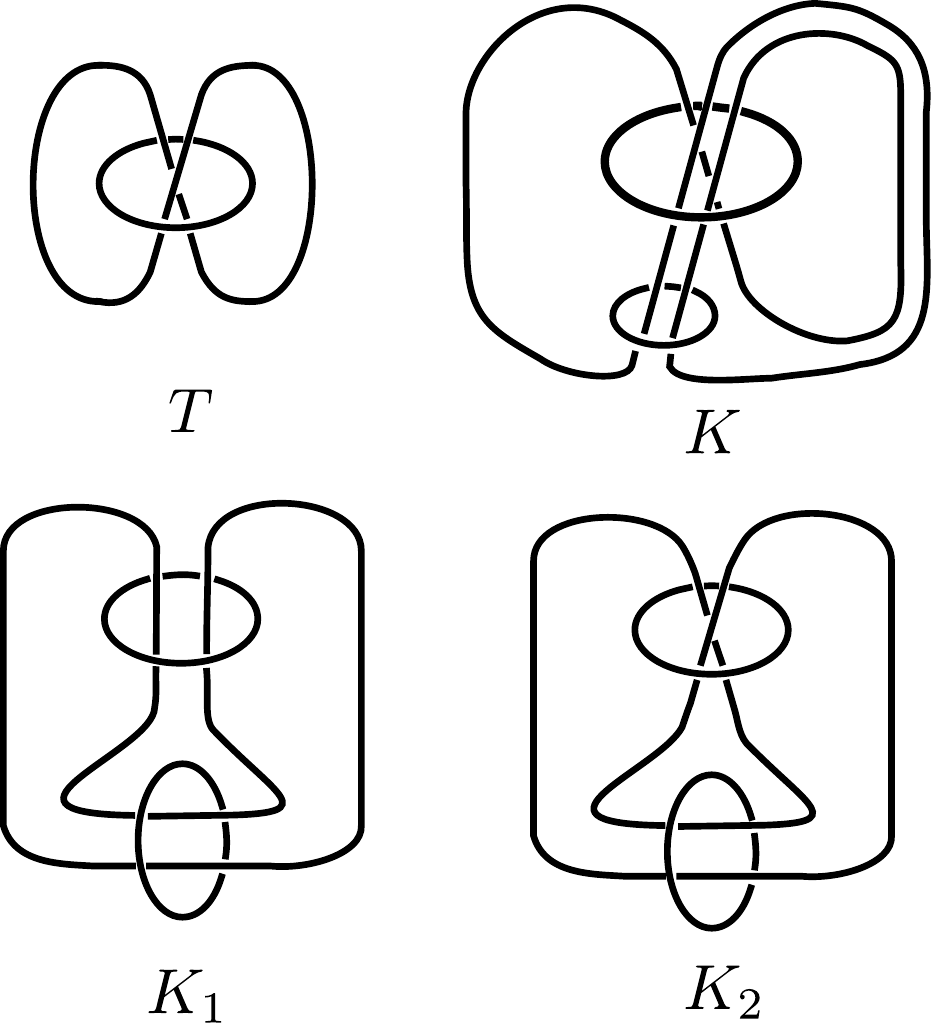}
\caption{The seed links 
$T = L4a1 = 4^2_1 = T(2,4)$ torus link, 
$K_1=L6a4=6^3_2=t12067$, $K=L10n84=10^3_{19}=t12039$ and 
$K_2=L8n5=8^3_9=t12066$.} 
\label{fig.Seeds}
\end{center}
\end{figure}

\begin{lemma}
\label{lem.links}
Up to mirror image, we have
$$
\calS_1=\{T\}, \qquad \calS_2= \{K_1,K_2,K\}
$$
where $T,K_i,K$ are the links shown in Figure~\ref{fig.Seeds}. 
\end{lemma}

\begin{proof}
The seed link $T$ is obtained from the theta graph by a single 
$X$ move. The links $K_1$ and $K_2$ are obtained by first doing an $A$ move 
to get a tetrahedron graph and then applying two $U's$ or a $U$ and an $X$ 
on a pair of disjoint edges. Finally $K$ is obtained from the tetrahedron by 
doing one $X$ move and then a $U$ move on one of the edges newly created by 
the $X$. One checks that all other sequences with at most one $A$ move
either give links with homeomorphic complement or links including two 
components that are not belts.
\end{proof}

$T(m)$ is the well-understood torus knot $T(2,2m+1)$. 
Observe that $K$ is the seed link of the fusion knots $K(m_1,m_2)$. 
$K_1(m_1,m_2)$ and $K_2(m_1,m_2)$ are alternating double-twist knots (with
an even or odd number of half-twists) that appear 
in~\cite{HS}. The Slope Conjecture is known for alternating knots~\cite{Ga1}. 
In particular, the Jones slopes are integers. 

The next lemma which can be proved using \cite{snappy} summarizes the 
hyperbolic geometry of the seed links $K_1$ and $K$. 

\begin{lemma}
\label{lem.2}
Each of the links $K_1$ and $K$ is obtained by face-pairings of 
two regular ideal octahedra. $K_1$ and $K$ are scissors congruent with volume 
$7.327724753\dots$, commensurable with a common $4$-fold cover, and have a 
common orbifold quotient, the Picard orbifod $H^3/\mathrm{PSL}(2,\BZ[i])$.
%
%
%
%
\end{lemma}

%

\subsection{The topology and geometry of the 2-fusion knots 
$K(m_1,m_2)$}
\label{sub.2fusiontop}

In this section we summarize what is known about the topology and geometry
of 2-fusion knots. The section is independent of the results of our paper,
and we include it for completeness. 

The 2-parameter family of 2-fusion knots specializes to 
\begin{itemize}
\item
The 2-strand torus knots by
$K(m_1,0)=T(2,2m_1+1)$.
\item
The non-alternating pretzel knots by $K(m_1,1)=(-2,3,2m_1+3)$ pretzel.
In particular, we have:
\begin{align*}
K(2,1)&=(-2,3,7) & K(1,1)&=(-2,3,5)=10_{124} & 
K(0,1)&=(-2,3,3)=8_{19} \\ 
K(-1,1)&=(-2,3,1)=5_1 &
K(-2,1)&=(-2,3,-1)=5_2 & K(-3,1)&=(-2,3,-3)=8_{20}.
\end{align*}
\item
Gordon's knots that appear in exceptional Dehn surgery~\cite{GW}. More
precisely, if $L^{GW}_2$ and $L^{GW}_3$ denote the
two 2-component links that appear in \cite[Fig.24.1]{GW}, then 
$L^{GW}_2(n)=K(-1,n)$. 
These two families intersect at the 
$(-2,3,7)$ pretzel knot; see also \cite[Fig.26]{E-M}.
Moreover, the knot $K(-1,3)=K4_3$ (following the notation of the 
census~\cite{snappy}) was the focus of \cite{GL2}. 
\end{itemize}
We thank Cameron Gordon for pointing out to us these specializations.

The next lemma summarizes some topological properties of the family 
$K(m_1,m_2)$. 

\begin{lemma}
\label{lem.top}
\rm{(a)}
$K(m_1,m_2)$ is the closure of the 3-string braid $\b_{m_1,m_2}$, where
$$
\b_{m_1,m_2}=b a^{2m_1+1}(ab)^{3 m_2}
$$
where $s_1=a, s_2=b$ are the standard generators of the braid group.
\newline
\rm{(b)}
$K(m_1,m_2)$ is a twisted torus knot obtained from the torus knot 
$T(3,3m_2+1)$ by applying $m_1$ full twists on two strings.
\newline
\rm{(c)}
$K(m_1,m_2)$ is a tunnel number $1$ knot, hence it is strongly invertible.
See~\cite{Lee} and also~\cite[Fact 1.2]{MSY}.
\newline
\rm{(d)} 
We have involutions
\begin{equation}
\label{eq.involution2}
K(m_1,m_2)=-K(1-m_1,-1-m_2), \qquad
K(-1,m_2)=K(-1,-m_2)
\end{equation} 
\rm{(e)}
$K(m_1,m_2)$ is hyperbolic when $m_1 \neq 0, 1$ and $m_2 \neq 0,-1$. 
\end{lemma}
The proof of part (e) follows by applying the 6-theorem~\cite{Agol,Lackenby}.

The next remark points out that the knots $K(m_1,m_2)$ are not always
Montesinos, nor alternating, nor adequate. So, it is a bit of a surprise that
one can compute some boundary slopes using the incompressibility criterion
of~\cite{DG} (this can be done for all integer values of $m_1,m_2$), and 
even more, that we can compute the Jones slope in Theorem~\ref{thm.1} 
and verify the Slope Conjecture. Thus, our methods apply beyond the class 
of Montesinos or alternating knots. 

\begin{remark}
\label{rem.nomontesinos}
$K(m_1,m_2)$ is not always a Montesinos knot. Indeed, recall that
the $2$-fold branched cover of a Montesinos
knot is a Seifert manifold~\cite{Montesinos}, in particular not hyperbolic. 
On the other hand, {\tt SnapPy}~\cite{snappy}
confirms that the $2$-fold branched cover of $K(-1,-3)$ (which appears 
in~\cite{GL2}) is a hyperbolic 
manifold, obtained by $(-2,3)$ filling of the sister {\tt m003} of the 
$4_1$ knot. 
\end{remark}

\subsection*{Acknowledgment}
S.G. was supported in part by NSF. R.V. was supported by the 
Netherlands Organization for Scientific Research. 
An early version of a manuscript by the first author
was presented in the Hahei Conference in 
New Zealand, January 2010. The first author wishes to thank Vaughan Jones 
for his kind invitation and hospitality and Marc Culler, Nathan Dunfield and 
Cameron Gordon for many enlightening conversations.


\appendix

\section{Sample values of the colored Jones function of 
$K(m_1,m_2)$}
\label{sec.values}

In this section we give some sample values of the colored Jones function
$J_{K(m_1,m_2),n}(q)$ which were computed using Theorem \ref{thm.cjk}
after a global change of $q$ to $1/q$. These values agree with independent 
calculations of the colored Jones function using the {\tt ColouredJones}
function of the {\tt KnotAtlas}
program of \cite{B-N}, confirming the consistency of our formulas with
KnotAtlas. This is a highly non-trivial check since KnotAtlas
and Theorem \ref{thm.cjk} are completely different formulas of the same
colored Jones polynomial. Here, $J_{K,n}(q)$ is normalized to be $1$
for the unknot (and all $n$) and $J_{K,1}(q)$ is the usual Jones polynomial 
of $K$. 

{\small  
\def\arraystretch{1.5}
\begin{center}
\begin{tabular}{|c|l|} \hline
$n$ & $J_{K(2,1),n}(q)$ 
\\ \hline
$0$ & $1$ 
\\ \hline
$1$ & 
$
q^5+q^7-q^{11}+q^{12}-q^{13}
$
\\ \hline
$2$ & 
$
q^{10}+q^{13}+q^{16}-q^{17}+q^{19}-q^{20}-q^{21}+q^{22}-q^{24}+q^{26}-q^{27}-q^{28}
+2 q^{29}-q^{30}-2 q^{31}+3 q^{32}
$
\\ & 
$
-q^{33}-2 q^{34}+2 q^{35}
$
\\ \hline
$3$ & 
$
q^{15}+q^{19}+q^{23}-q^{25}+q^{27}-q^{29}-q^{33}+q^{34}-2 q^{37}+q^{38}+q^{39}
-2 q^{41}+q^{43}+q^{44}-q^{45}-q^{46}
$
\\ & 
$
+q^{48}+q^{49}-2 q^{50}-q^{51}+q^{52}+2 q^{53}-2 q^{54}-2 q^{55}+2 q^{56}+3 q^{57}
-2 q^{58}-3 q^{59}+3 q^{60}+3 q^{61}
$
\\ & 
$
-2 q^{62}-3 q^{63}+q^{64}+3 q^{65}-q^{66}-q^{67}
$
\\ \hline
$4$ & 
$
q^{20}+q^{25}+q^{30}-q^{33}+q^{35}-q^{38}+q^{40}-q^{41}-q^{43}+2 q^{45}-q^{46}
-q^{48}-q^{49}+2 q^{50}-q^{51}+q^{52}
$
\\ & 
$
-q^{53}-q^{54}+2 q^{55}-2 q^{56}+q^{57}-q^{58}+3 q^{60}-2 q^{61}-2 q^{63}+3 q^{65}
-3 q^{68}-q^{69}+2 q^{70}+q^{71}
$
\\ & 
$
+q^{72}-2 q^{73}
-2 q^{74}+2 q^{76}+2 q^{77}-2 q^{79}-2 q^{80}+2 q^{81}+2 q^{82}+2 q^{83}-2 q^{84}
-4 q^{85}+2 q^{86}
$
\\ & 
$
+2 q^{87}+3 q^{88}-2 q^{89}
-6 q^{90}+3 q^{91}+2 q^{92}+4 q^{93}-3 q^{94}-7 q^{95}+4 q^{96}+2 q^{97}+4 q^{98}
-2 q^{99}
$
\\ & 
$
-7 q^{100}+2 q^{101}+q^{102}
+4 q^{103}-q^{104}-4 q^{105}+q^{106}+q^{108}
$
\\ \hline
\end{tabular}
\end{center}
}

{\small  
\def\arraystretch{1.5}
\begin{center}
\begin{tabular}{|c|l|} \hline
$n$ & $J_{K(1,3),n}(q)$ 
\\ \hline
$0$ & $1$ 
\\ \hline
$1$ & 
$
q^{10}+q^{12}-q^{22}
$
\\ \hline
$2$ & 
$
q^{20}+q^{23}+q^{26}-q^{27}+q^{29}-q^{30}+q^{32}
-q^{33}+q^{35}-q^{36}+q^{38}-q^{39}+q^{41}-q^{42}-q^{43}+q^{44}
$
\\ & 
$
-q^{45}
-q^{46}+q^{47}-q^{49}+q^{50}-q^{52}+q^{53}-q^{55}+q^{56}-q^{58}+q^{59}-q^{61}+q^{62}
-q^{64}+q^{65}
$
\\ \hline
$3$ & 
$
q^{20}+q^{23}+q^{26}-q^{27}+q^{29}-q^{30}+q^{32}
-q^{33}+q^{35}-q^{36}+q^{38}-q^{39}+q^{41}-q^{42}-q^{43}+q^{44}
$
\\ & 
$
-q^{45}
-q^{46}+q^{47}-q^{49}+q^{50}-q^{52}+q^{53}-q^{55}+q^{56}-q^{58}+q^{59}-q^{61}+q^{62}
-q^{64}+q^{65}
$
\\ \hline
$4$ & 
$
q^{40}+q^{45}+q^{50}-q^{53}+q^{55}-q^{58}+q^{60}-q^{63}+q^{65}-q^{68}+q^{70}-q^{73}
+q^{75}-q^{78}+q^{80}-q^{83}
$
\\ & 
$
-q^{88}
-q^{93}+q^{96}-q^{98}+q^{101}-q^{103}+q^{106}-q^{108}
+q^{111}-q^{113}+q^{116}-q^{118}+q^{121}-q^{123}
$
\\ & 
$
+q^{126}-q^{128}
+q^{131}-q^{133}+q^{136}
-q^{138}+q^{141}-q^{143}+q^{146}-q^{148}+q^{151}+q^{156}-q^{160}+q^{161}
$
\\ & 
$
-q^{165}+q^{166}
-q^{170}+q^{171}-q^{175}-q^{180}+q^{181}-q^{185}+q^{186}-q^{190}+q^{191}-q^{195}+q^{196}
$
\\ \hline
\end{tabular}
\end{center}
}

{\small  
\def\arraystretch{1.5}
\begin{center}
\begin{tabular}{|c|l|} \hline
$n$ & $J_{K(-2,3),n}(q)$ 
\\ \hline
$0$ & $1$ 
\\ \hline
$1$ & 
$
q^7+q^9-q^{14}+q^{15}-q^{16}+q^{17}-q^{18}
$
\\ \hline
$2$ & 
$
q^{14}+q^{17}+q^{20}-q^{21}+q^{23}-q^{24}+q^{26}-2 q^{27}+2 q^{29}-2 q^{30}-q^{31}
+3 q^{32}-q^{33}-2 q^{34}+2 q^{35}
$
\\ & 
$
-q^{37}-q^{41}+q^{42}-q^{44}+q^{46}-q^{48}+q^{49}
$
\\ \hline
$3$ & 
$
q^{21}+q^{25}+q^{29}-q^{31}+q^{33}-q^{35}+q^{37}-q^{39}-q^{40}+q^{41}-q^{44}+q^{45}
-q^{46}+2 q^{49}-q^{50}-q^{51}
$
\\ & 
$
-q^{52}+2 q^{53}-q^{55}+q^{57}-q^{58}-2 q^{59}+2 q^{60}+3 q^{61}
-3 q^{62}-4 q^{63}+2 q^{64}+6 q^{65}-2 q^{66}-7 q^{67}
$
\\ & 
$
+q^{68}+6 q^{69}+q^{70}-7 q^{71}
-q^{72}+7 q^{73}+2 q^{74}-7 q^{75}-2 q^{76}+7 q^{77}+2 q^{78}-7 q^{79}-2 q^{80}
+7 q^{81}
$
\\ & 
$
+3 q^{82}-6 q^{83}-2 q^{84}+3 q^{85}+2 q^{86}-q^{87}-q^{88}+q^{92}-q^{93}
$
\\ \hline
$4$ & 
$
q^{28}+q^{33}+q^{38}-q^{41}+q^{43}-q^{46}+q^{48}-q^{51}-q^{56}+q^{57}-2 q^{61}+q^{62}
+q^{64}+q^{65}-2 q^{66}+q^{67}
$
\\ & 
$
-q^{68}+q^{70}-2 q^{71}+2 q^{72}-q^{73}+2 q^{75}-3 q^{76}
+2 q^{77}-2 q^{78}-q^{79}+3 q^{80}-2 q^{81}+4 q^{82}-2 q^{83}
$
\\ & 
$
-3 q^{84}+q^{85}-3 q^{86}
+5 q^{87}+q^{88}-q^{89}-6 q^{91}+3 q^{92}+q^{93}+2 q^{95}-3 q^{96}+2 q^{97}-q^{98}
-3 q^{99}
$
\\ & 
$
+q^{100}+q^{101}+6 q^{102}+q^{103}-8 q^{104}-5 q^{105}+2 q^{106}+11 q^{107}
+6 q^{108}-10 q^{109}-10 q^{110}-q^{111}
$
\\ & 
$
+13 q^{112}+11 q^{113}-10 q^{114}-13 q^{115}
-4 q^{116}+15 q^{117}+14 q^{118}-10 q^{119}-15 q^{120}-4 q^{121}
$
\\ & 
$
+15 q^{122}+15 q^{123}
-11 q^{124}-16 q^{125}-3 q^{126}+15 q^{127}+15 q^{128}-10 q^{129}-16 q^{130}-5 q^{131}
$
\\ & 
$
+14 q^{132}+15 q^{133}-6 q^{134}-12 q^{135}-8 q^{136}+7 q^{137}+9 q^{138}-3 q^{140}
-6 q^{141}+q^{142}+3 q^{143}
$
\\ & 
$
+q^{145}-q^{146}-q^{149}+q^{150}
$
\\ \hline
\end{tabular}
\end{center}
}


\bibliographystyle{hamsalpha}
\bibliography{biblio}
\end{document}